\documentclass[11pt,letterpaper]{amsart}
\usepackage[top=3.75cm, bottom=3cm, left=3cm, right=3cm]{geometry}
\usepackage{amsmath,amsthm,amssymb,amsfonts,extarrows,mathrsfs,comment,tikz-cd}
\usepackage{stmaryrd}
\usepackage{dsfont}
\usepackage{mathtools}
\usepackage{enumerate}
\usepackage[all]{xy}
\usepackage{indentfirst}
\usepackage{mathtools}
\usepackage{hyperref}
\usepackage{fancyhdr}
\usepackage{graphicx,epstopdf}

\theoremstyle{plain}
\newtheorem{theorem}{Theorem}[section]
\newtheorem{corollary}[theorem]{Corollary}
\newtheorem{lemma}[theorem]{Lemma}
\newtheorem{proposition}[theorem]{Proposition}
\newtheorem{conjecture}[theorem]{Conjecture}

\newtheorem{speculation}[theorem]{Speculation}
\theoremstyle{definition}
\newtheorem{definition}[theorem]{Definition}
\newtheorem{example}[theorem]{Example}

\theoremstyle{remark}
\newtheorem{remark}[theorem]{Remark}

\newcommand\cA{\mathcal{A}}

\newcommand\cU{\mathcal{U}}

\newcommand\CC{\mathbb{C}}

\newcommand\GG{\mathbb{G}}

\newcommand\QQ{\mathbb{Q}}

\newcommand\ZZ{\mathbb{Z}}

\newcommand\bfe{\mathbf{e}}

\DeclareMathOperator{\Br}{Br}

\DeclareMathOperator{\id}{id}

\DeclareMathOperator{\Pic}{Pic}

\title[A twisted derived category]{A twisted derived category of hyper-Kähler varieties of $K3^{[n]}$-type}

\date{\today}

\begin{document}
\author[R. Zhang]{Ruxuan Zhang}
\address{Fudan University}
\email{rxzhang18@fudan.edu.cn}

\begin{abstract}
We conjecture that a natural twisted derived category of any hyper-K\"ahler variety of $K3^{[n]}$-type is controlled by its Markman-Mukai lattice. We prove the conjecture under numerical constraints, and our proof relies on Markman's projectively hyperholomorphic bundle and a recently proven twisted version of the D-equivalence conjecture.

In particular, we prove that any two fine moduli spaces of stable sheaves on a $K3$ surface are derived equivalent if they have the same dimension. 
\end{abstract}
\maketitle
\setcounter{tocdepth}{1}

\tableofcontents

\section{Introduction}
 Throughout,we work over the field of complex numbers $\CC$.
Let $X$ be any hyper-K\"ahler variety of $K3^{[n]}$-type with $n\geq 2$. A general philosophy predicts that $X$ is a moduli space of stable objects on some $K3$ category $\cA$, which is expected to be governed by the Markman-Mukai lattice
$L(X)$.

A classical example supporting this perspective arises when $X$ is the Hilbert scheme of $n$ points on a $K3$ surface $S$.
In this case, the associated $K3$ category is simply $D^b(S)$ and
there exists a Hodge isometry \[L(X)\cong \widetilde{H}(S,\ZZ)\]
which identifies the Markman-Mukai lattice $L(X)$ with the Mukai lattice $\widetilde{H}(S,\ZZ)$ of $S$.
By the derived Torelli theorem, the bounded derived category $D^b(S)$ is determined by the Mukai lattice. Furthermore, the Bridgeland-King-Reid (BKR) equivalence implies that 
$D^b(X)$ is completely determined by $D^b(S)$. Consequently, $D^b(X)$ is also controlled by the Mukai lattice $\widetilde{H}(S,\ZZ)$. A more recent example of this phenomenon can be found in \cite{BH24}.

In the general setting, the existence of the $K3$ category $\cA$ remains an open question.
Nevertheless, one may still ask whether the derived category $D^b(X)$ is controlled by the Markman-Mukai lattice $L(X)$. 
If true, this would provide supporting evidence for both the existence of $\cA$ and the possibility of constructing 
$D^b(X)$ from $\cA$. However, a direct attempt at establishing this connection does not hold in full generality—see \cite[Theorem 5.13]{mattei2024obstruction} for a counterexample. 
Instead, we propose a refinement by considering a natural twisted derived category of $X$. To be more precise, we have the following speculation:
\begin{speculation}\label{spec:main}
    Let $X,~Y$ be hyper-K\"ahler varieties of $K3^{[n]}$-type with $n\geq 2$.
    \begin{enumerate}
        \item There exists a $K3$ category $\cA_X$ such that $X$ is a moduli space of stable objects in $\cA_X$.
        \item 
        If there is a Hodge isometry $L(X)\cong L(Y)$, then we have \[\cA_X\cong \cA_Y {\rm ~or~}\cA_X\cong \cA_Y^{\rm op}\]
        \item There exists a canonical Brauer class $\theta_X$  such that \[{\rm Sym}^n\cA_X\cong D^b(X,\pm n\theta_X).\]
    \end{enumerate}
\end{speculation}
The existence of the $K3$ category has been studied by Markman and Mehrotra in \cite{markman2015integral} and by Chunyi Li, E.Macri, A.Perry, P. Stellari, and Xiaolei Zhao recently.  

We observe that $X$ naturally arises as a moduli space of stable objects not only in the category $\cA_X$ but also in its opposite $\cA_X^{\rm op}$. 
In this paper, we define a natural orientation on $L(X)$ to distinguish between the conjectural category $\cA_X$ and $\cA_X^{\rm op}$. Additionally, we define a natural Brauer class on $X$, which is predicted to be exactly the obstruction preventing $X$ from being a fine moduli space on $\cA_X$. 

We shall briefly introduce the definition of orientation and the Brauer class here.
The orthogonal complement $H^2(X)^\perp \subset L(X)$ is a rank-one lattice generated by a primitive vector whose square is $2n-2$. 
Choosing a generator $v$ (noting that $-v$ is also a generator) induces a natural orientation on $L(X)$. We denote the Markman-Mukai lattice with this orientation by $(L(X),v)$.
Associated with this choice of $v$ is a canonical cohomology class
\[\theta_{v} \in H^2(X,\mu_{2n-2}),\] which is invariant under the discriminant-preserving subgroup of the monodromy group.  Moreover, under a change of orientation, we have $\theta_{-v}=-\theta_v$.
For details, see Section \ref{sec:lattice}. The class $\theta_{v}$ has essentially appeared in \cite[Lemma 7.4]{markman2020beauville}. 

The presence of the canonical class $\theta_v$ implies  that any hyper-K\"ahler variety of $K3^{[n]}$-type naturally carries natural $\mu_{2n-2}$-gerbe structures. This motivates the study of the derived category of twisted sheaves on such gerbes.
In particular, when $X$ is a moduli space of stable sheaves on a $K3$ surface, the Brauer class $[\theta_{v}] \in \Br(X)$ corresponds to the obstruction preventing $X$ from being a fine moduli space (see Lemma \ref{lem:obs} for details).

We now state the main conjecture of this work: the Markman-Mukai lattice $L(X)$ governs the derived category of twisted sheaves on a natural gerbe. This conjecture is both a direct consequence of Speculation \ref{spec:main} and an explicit realization of its broader framework.
\begin{conjecture}\label{conj:main}
Let $X$ and $Y$ be two hyper-K\"ahler varieties of $K3^{[n]}$-type. If there exists a Hodge isometry between oriented Markman-Mukai lattices
\[\phi:(L(X),v)\rightarrow (L(Y),w),\] then there is an equivalence of bounded derived categories
    \[D^b(X,\left[n\theta_{v}\right])\cong D^b(Y,\left[\epsilon n\theta_{ w}\right]).\]
Here $\epsilon=1$ if $\phi$ is orientation-preserving, and $\epsilon=-1$ if $\phi$ is orientation-reversing.
\end{conjecture}
It is important to note that the identity map of lattices  $(L(X),v)\rightarrow (L(X),-v)$ reverses  orientations. 
Hence, we can reduce to the orientation-preserving case.

A supporting case of this conjecture is the twisted D-equivalence theorem, recently established in \cite{MSYZ24}.
This case also forms the foundation of our approach: reducing the conjecture to lattice-theoretic computations of Markman’s hyperholomorphic bundles \cite{M24}.

The first part of our results provides lattice-theoretic evidence in support of Conjecture \ref{conj:main}. In particular, the Brauer classes $[n\theta_v],~[n\theta_w]$ naturally arise in the relevant lattice computations (see Section \ref{sec:lattice} for details).
Among the possible $B$-field lifts of $\theta_v$, certain choices exhibit more desirable properties. Such a preferred lift is given by \[\frac{\delta_v}{2n-2}\in H^2(X,\QQ)\] for a distinguished class $\delta_v$.
We introduce a twisted extended Mukai lattice (with a natural orientation) and obtain
\begin{theorem}\label{thm:lattice}
    Let $X,~Y,~\phi$ be as in Conjecture \ref{conj:main}. Then there is a Hodge isometry \[\psi: \widetilde{H}(X,\frac{\delta_v}{2n-2},\ZZ)\cong \widetilde{H}(Y, \frac{\delta_w}{2n-2},\ZZ).\] Moreover, $\psi$ is orientation-preserving if and only if $\phi$ is.
\end{theorem}
The twisted extended Mukai lattice up to Hodge isometry is independent of the specific choice of a $B$-field. We prove a derived Torelli-type theorem (Theorem \ref{thm:criterion}) by constructing  Markman's projectively hyperholomorphic bundle and applying the twisted derived equivalence for birational hyper-K\"ahler varieties of $K3^{[n]}$-type. This applies to the following result:
\begin{theorem}\label{thm:primitive-embedding}
Conjecture \ref{conj:main} holds if  \[{\rm Span}\left(\phi(v),w\right)\subset L(Y)\] is a primitive lattice embedding. Here, ${\rm Span}\left(\phi(v),w\right)$ denotes the sublattice generated by integral linear combinations of $\phi(v)$ and $w$.
\end{theorem}
Next, we examine moduli spaces of stable objects on a $K3$ surface $S$. Let $v$ be a primitive vector in $\widetilde{H}^{1,1}(S,\ZZ)$, and denote by $M_v$ the moduli space of stable objects with respect to any generic stability condition.\footnote{Since moduli spaces associated with different stability conditions are birational, their twisted derived categories are equivalent by \cite{MSYZ24}; see Subsection \ref{subsec:birational} for details.}
Theorem \ref{thm:primitive-embedding} has the following consequence:
\begin{theorem}\label{thm:k3}
     Let $S$ be a projective $K3$ surface and $w$ be a primitive vector in $\widetilde{H}^{1,1}(S,\ZZ)$ satisfying $w^2=2n-2\geq2$. Then there exist derived equivalences \[ D^b(M_w,[ n\theta_{w}])\cong D^b(M_w,[- n\theta_{w}])\cong D^b(S^{[n]}).\] 
\end{theorem}
If $v^2=2$ or if $M_w$ is a fine moduli space of stable sheaves, then the Brauer classes $[\pm n\theta_{{w}}]$ vanish. In these cases, we obtain derived equivalences between untwisted categories, thereby providing an affirmative answer to a question posed by Huybrechts:
\begin{corollary}\label{cor:fine}
    Let $S$ be a projective $K3$ surface and $v,w$ be primitive vectors in $\widetilde{H}^{1,1}(S,\ZZ)$ with $v^2=w^2=2n-2\geq 2$. 
   If either $n=2$ or both $M_v$ and $M_w$ are birational to fine moduli spaces of stable sheaves, then there exists a derived equivalence \[D^b(M_v)\cong D^b(M_w).\]
\end{corollary}
We remark that $M_v$ (for any generic stability condition) is birational to a fine moduli space of stable sheaves if and only if $\mathrm{div}(v)=1$ in $\tilde{H}^{1,1}(S,\ZZ)$ as proved in \cite[Corollary 4.6]{mattei2024obstruction}.
This property depends solely on the Mukai vector $v$.

\subsection*{Acknowledgements}
We are grateful to Zhiyuan Li, Hanfei Guo, Dominique Mattei, Reinder Meinsma, Alex Perry, Junliang Shen, and Qizheng Yin for helpful discussions. The author was supported by the NKRD Program of China No.~2020YFA0713200 and LNMS.

\section{Lattice-theoretical evidence}\label{sec:lattice}
Let $X$ be a hyper-K\"ahler variety of $K3^{[n]}$-type, where $n\geq 2.$
In this section, we recall the precise definition of the class $\theta_{v} \in H^2(X,\mu_{2n-2})$ and introduce the twisted extended Mukai lattice.
When $X$ and $Y$ satisfy the assumptions of Conjecture \ref{conj:main}, we establish a Hodge isometry between their corresponding twisted extended Mukai lattices, thereby providing lattice-theoretical evidence in support of the conjecture.
\subsection{Oriented  Markman-Mukai lattice}\label{subsec:oriented}
The second cohomology group $H^2(X,\ZZ)$ 
carries the Beauville-Bogomolov-Fujiki (BBF) form, along with a weight-2 integral Hodge structure. Markman has described an extension of lattice and Hodge structure, denoted by $H^2(X,\ZZ)\subset L(X)$, which we refer to as the \emph{Markman-Mukai lattice}. This lattice satisfies the following properties (see~\cite[Corollary 9.5]{Markman11}\cite[Section 1]{BHT15}):
\begin{enumerate}
    \item As a lattice, $L(X)\cong U^{\oplus 4}\oplus E_8(-1)^{\oplus 2}$, where it has signature $(4,20)$. The weight-2 Hodge structure on $L(X)$ is inherited form $H^2(X,\ZZ)$.
    \item The orthogonal complement $H^2(X,\ZZ)^\perp \subset L(X)$ is generated by a primitive vector of square $2n-2.$
    \item If $X$ is a moduli space of stable objects on a $K3$ surface $S$ with Mukai vector $v\in \widetilde{H}^{1,1}(S,\ZZ)$, then the inclusion $H^2(X,\ZZ)\subset L(X)$ is identified with $v^\perp\subset \widetilde{H}(S,\ZZ)$.
    \item An isometry $\rho: H^2(X_1,\ZZ)\to H^2(X_2,\ZZ)$ is a parallel transport operator if and only if $\rho$ is orientation-preserving and can be lifted to an isometry $\phi:L(X_1)\rightarrow L(X_2)$.
 \end{enumerate}
\begin{remark}
   Markman's construction establishes a natural choice of $O(L(X))$-orbit of embeddings of $H^2(X)$ to $L(X)$,  where the latter is considered without additional Hodge structure. Therefore, the embedding $H^2(X)\subset L(X)$ and the Hodge structure on $L(X)$ described in \cite{BHT15} are only canonical up to isometries of $L(X)$. All the above properties hold for any embedding in the natural $O(L(X))$-orbit.
\end{remark}

To define an orientation on $L(X)$, we fix a generator $v$ of $H^2(X,\ZZ)^\perp$ and proceed as follows.
Let $\sigma$ be the $2$-form on $X$, and let $h\in H^2(X)$ be a K\"ahler class. Then there is a positive four-dimensional real subspace in $L(X)$ spanned by \[\mathrm{Re}(\sigma),~\mathrm{Im}(\sigma),~h,~v.\] 
This basis determines an orientation of $L(X)$, which is independent of the choice of the K\"ahler class $h$. We denote this oriented lattice by $(L(X),v)$ and refer to it as the \emph{oriented Markman-Mukai lattice}.

Now, let $Y$ be another hyper-K\"ahler variety of $K3^{[n]}$-type. A Hodge isometry between oriented Markman-Mukai lattices \[\phi:(L(X),v)\rightarrow (L(Y),w)\] is simply  a Hodge isometry $\phi:L(X)\rightarrow L(Y).$. 
We classify  $\phi$ as orientation-preserving if it preserves the orientations determined by $v$ and $w$, and as orientation-reversing otherwise.  We remark that the identity map of lattices $\phi:(L(X),v)\rightarrow (L(X),-v)$ is orientation-reversing.



\subsection{The canonical class in  $H^2(X,\mu_{2n-2})$ }
Assume that $X$ is a manifold of $K3^{[n]}$-type. Since $X$ has no odd cohomology, the discussion in \cite[Section 4.1]{Ca02} yields the following explicit description of the (cohomological) Brauer group:
\begin{equation}\label{Br}
\mathrm{Br}(X) = \left( H^2(X, \ZZ) / \mathrm{Pic}(X)\right) \otimes \QQ/\ZZ.
\end{equation}
 The description (\ref{Br}) also allows us to present a Brauer class in the form
\begin{equation}\label{B_Field}
\left[ \frac{\beta}{d}\right] \in \mathrm{Br}(X), \quad \beta\in H^2(X,\ZZ), \quad d\in \ZZ \setminus \{0\};
\end{equation}
this is referred to as the ``$B$-field''.
Since $X$ has no odd-degree cohomology, we obtain the following short exact sequence:
\begin{equation}\label{ses:1}
    0\rightarrow  H^2(X,\ZZ)\xrightarrow{\times (2n-2)} H^2(X,\ZZ) \xrightarrow{\pi} H^2(X,\mu_{2n-2})\rightarrow 0.
\end{equation}
The natural inclusion $\mu_{2n-2}\hookrightarrow \GG_m$ induces a  map \[[-]:H^2(X,\mu_{2n-2})\rightarrow \mathrm{Br}(X).\] Here by abuse of notation, we also denote this map by $[-]$ and we have \[\left[\pi(\beta)\right]=\left[\frac{\beta}{2n-2}\right]\in \mathrm{Br}(X).\]
We introduce a class $\delta_v\in H^2(X,\ZZ)$ and show that the image $\pi(\delta_v)\in H^2(X,\mu_{2n-2})$ is a canonical class.
It leads to a class $\left[\pi(\delta_v)\right]=\left[\frac{\delta_v}{2n-2}\right]\in \mathrm{Br}(X)$.

The class $\left[\frac{\delta_v}{2n-2}\right]$ has already appeared in \cite[Lemma 7.4]{markman2020beauville} and we now reformulate it as follows:
\begin{lemma}\label{lem:delta}
       Let $v$ be a generator of $H^2(X,\ZZ)^\perp\subset L(X)$. Then there exists a class (non-unique) \[\delta_{v} \in H^2(X,\ZZ)\] satisfying the following conditions:
   \begin{enumerate}
       \item The square $\delta^2_v=2-2n$ and the divisibility $\mathrm{div}(\delta_v)=2n-2$ in $H^2(X,\ZZ)$.
       \item The class $\frac{\delta_v-v}{2n-2}$ is integral in $L(X)$.
   \end{enumerate}
   Moreover, the image $\pi(\delta_v)$ in $H^2(X,\mu_{2n-2})$ via the sequence \eqref{ses:1} is independent of the choice of $\delta_v$ and is invariant under the subgroup of the monodromy group that preserves the discriminant. 
\end{lemma}
\begin{proof}
     For conditions  (1) and (2), we apply a parallel transport operator $\rho$ to a Hilbert scheme of $n$ points on a $K3$ surface $S$, yielding the following commutative diagram:
    \begin{equation*}
        \begin{tikzcd}
{H^2(X,\ZZ)} \arrow[r, "\rho"] \arrow[d, hook'] & {H^2(S^{[n]},\ZZ)} \arrow[d, hook'] \\
L(X) \arrow[r, "\phi"]                            & \widetilde{H}(S,\ZZ)                                 
\end{tikzcd}.
\end{equation*}
We then define $\delta_v$ as follows:
\begin{itemize}
    \item If $\phi(v)=(1,0,1-n)$, then take $\delta_v=\phi^{-1}((1,0,n-1))$.
    \item If $\phi(v)=(-1,0,n-1)$, then take $\delta_v=\phi^{-1}((-1,0,1-n))$.
\end{itemize}
For the final statement, the uniqueness follows directly from (2). Any monodromy operator sends $\pi(\delta_v)$ to $\pm \pi(\delta_v)$.
 When $n=2$, we have \[\pi(\delta_v)=-\pi(\delta_v)\in H^2(X,\mu_{2})\] and there is nothing to prove.
 When $n\geq 3$, any monodromy operator preserving the discriminant is necessarily lifted to an isometry $\phi:L(X)\to L(X)$, which satisfies $\phi(v)=v$. Therefore, the invariance is implied by the uniqueness.
\end{proof}

Now we can give the following definition:
\begin{definition}
We refer to any class satisfying Lemma \ref{lem:delta}(1)(2) as a $\delta$-class. 
We then define \[\theta_v:=\pi(\delta_v)\in H^2(X,\mu_{2n-2})\] 
for any $\delta$-class.
\end{definition}
We have $[\theta_v]=\left[\frac{\delta_v}{2n-2}\right]\in\mathrm{Br}(X)$ described as above.
Although neither $\delta_v$ nor $[\theta_v]$ is canonical, the class $\theta_v \in H^2(X,\mu_{2n-2})$ is canonical up to a choice of the orientation of $L(X)$. We remark that \[\theta_v=-\theta_{-v},\] follows immediately from the definition.

The result reads that \emph{any} hyper-K\"ahler variety of $K3^{[n]}$-type carries a natural $\mu_{2n-2}$-gerbe structure. 
While this $\mu_{2n-2}$-gerbe structure is always nontrivial, the corresponding $\GG_m$-gerbe may be trivial.
Markman has observed the following:
\begin{lemma}[\cite{markman2020beauville},\cite{mattei2024obstruction}]\label{lem:obs}
    Suppose $X$ is a moduli of stable sheaves on a $K3$ surface $S$.
    Then $X$ is fine if and only if $[\theta_v]$ is a trivial Brauer class. 
\end{lemma}
\begin{proof}
Markman proved in \cite[Lemma 7.5(2)]{markman2020beauville} that $[\theta_v]$ is trivial if and only if
\[\mathrm{div}(v)=1 {\rm ~in~}\tilde{H}^{1,1}(X,\ZZ),\]
which is equivalent to $X$ being a fine moduli space, according to
\cite[Corollary 4.6]{mattei2024obstruction}.
\end{proof}

\subsection{Twisted extended Mukai lattices}
Let $B\in H^2(X,\QQ)$ be a $B$-field. We define the twisted extended Mukai lattice $\widetilde{H}(X,B,\ZZ)$ as the $K3$ case, which is an oriented lattice with a weight-2 Hodge structure. This lattice is expected to control the derived category of certain twisted sheaves.
To be precise, as a lattice, we have \[\widetilde{H}(X,B,\ZZ)\cong H^2(X,\ZZ)\oplus U,\] where $U$ is the hyperbolic plane with generators $e,~f$, satisfying $\left<e.f\right>=-1,~e^2=f^2=0$. In this way, the BBF form on $H^2(X,\ZZ)$ extends to the Mukai pairing $\left<-.-\right>$.
The weight-2 Hodge structure on $\widetilde{H}(X,B,\ZZ)$ is given by \[\widetilde{H}^{2,0}(X,B)=\CC \cdot {\bfe}^B(\sigma)=\CC\cdot(\sigma+\left<B.\sigma\right>f),\] where $\sigma$ is the $2$-form on $X$ and 
\[\begin{aligned}
    {\bfe}^B:\widetilde{H}(X,\QQ)&\longrightarrow \widetilde{H}(X,\QQ)\\
    re+c+sf &\mapsto re+(c+rB)+(rB^2/2+s+\left<c.B\right>)f.
\end{aligned}
\] for any $c\in H^2(X,\QQ)$ and $r,~s\in \QQ$.
The orientation on $\widetilde{H}(X,B,\ZZ)$ is induced by the basis
\[\mathrm{Re}(\sigma),~\mathrm{Im}(\sigma),~\mathrm{Re}({\bfe}^{ih}(e)),~\mathrm{Im}({\bfe}^{ih}(e))\] 
of a four-dimensional positive space, where $h$ is a K\"ahler class.
When $B$ is integral, ${\bfe}^B$ is an integral isometry preserving both the discriminant and the orientation. In fact, for any even lattice $\Lambda$, one can define ${\bfe}^B\in O(\Lambda\oplus U)$ using the same formula.
There are other integral transformations that are useful, known as Eichler transvections. \footnote{Note that some of the signs differ from \cite{GHS09}, as $\langle e. f \rangle = -1$ in this case, while it is 1 in \cite{GHS09}.}
\begin{definition}
    Let $\Lambda$ be any even lattice.
    For an integral class $b\in \Lambda$, the \emph{Eichler transvection} $E_b$ is an integral isometry of $\Lambda\oplus U$ such that for any $c\in \Lambda$ and $r,~s\in\ZZ$, we have
    \[E_b(re+c+sf)=(r-\left<b.c\right>+\frac{sb^2}{2})e+(c-sb)+sf.\]
\end{definition}
\begin{lemma}
    The Eichler transvection $E_b$ preserves the discriminant and the orientation of $\Lambda \oplus U$. 
\end{lemma}
\begin{proof}
Let $\eta$ be the isometry exchanging $e,~f$ and acting as $\id$ on $\Lambda$.
Then, by a straightforward computation, we find \[E_b=\eta\circ {\bfe}^{-b}\circ \eta\] and the assertion follows.
\end{proof}
In this paper, the lattice $\Lambda$ is taken to be either $L(X)$ or $H^2(X,\ZZ)$.
\subsection{Proof of Theorem \ref{thm:lattice}}\label{subsec:proof-lattice}
Let $X,~Y$ be hyper-K\"ahler varieties of $K3^{[n]}$-type. Assume that there exists a Hodge isometry between the oriented Markman-Mukai lattices
\[\phi:(L(X),v)\rightarrow (L(Y),w).\] Since $v^2=w^2$, we have \begin{equation}\label{eq:2}
    \left<\phi(v).(\phi(v)-w)\right>=\frac{1}{2}(\phi(v)-w)^2.
\end{equation}
Now, consider the isometry \[E_{\tilde{\phi}(v)-w}\circ \tilde{\phi}:L(X)\oplus U\to L(Y)\oplus U,\] where $\tilde{\phi}:=\phi\oplus \id_U.$ 
A straightforward computation using the definition and equation \eqref{eq:2} shows that
\begin{equation}\label{eq:3}
    E_{\tilde{\phi}(v)-w}\circ \tilde{\phi}(v+f)=w+f.
\end{equation}
Take $\delta$-classes $\delta_v\in H^2(X,\ZZ)\subset L(X)$ and $\delta_w\in H^2(Y,\ZZ)\subset L(Y)$ as in Lemma \ref{lem:delta}. Then ${\bfe}^{\frac{\delta_v-v}{2n-2}}$ (resp. ${\bfe}^{\frac{\delta_v-v}{2n-2}}$)  is an integral isometry of $L(X)\oplus U$ (resp. $L(Y)\oplus U$) by Lemma \ref{lem:delta}(2). 
Moreover, we have \begin{equation}\label{eq:4}
    {\bfe}^{\frac{v-\delta_v}{2n-2}}(v)=v+f,~{\bfe}^{\frac{\delta_w-w}{2n-2}}(w+f)=w
\end{equation}
since $\left<\delta_v.v\right>=\left<\delta_w.w\right>=0$.
Combining \eqref{eq:3} and \eqref{eq:4}, we obtain an integral isometry 
\begin{equation}\label{eq:5}
    \tilde{\psi}:{\bfe}^{\frac{\delta_w-w}{2n-2}} \circ E_{\tilde{\phi}(v)-w}\circ \tilde{\phi}\circ  {\bfe}^{\frac{v-\delta_v}{2n-2}} :L(X)\oplus U\to L(Y)\oplus U
\end{equation}
satisfying
\begin{equation}\label{eq:6}
    \tilde{\psi}(v)=w.
\end{equation}
Since $ E_{\tilde{\phi}(v)-w}\circ \tilde{\phi}$ respects the untwisted Hodge structure, we can compute that
\begin{equation}\label{eq:7}
    \tilde{\psi}(\sigma_X+\left<\frac{\delta_v}{2n-2}.\sigma_X\right>f)=\sigma_Y+\left<\frac{\delta_w}{2n-2}.\sigma_Y\right>f.
\end{equation}
Combining \eqref{eq:6} and \eqref{eq:7}, we obtain the desired Hodge isometry
\begin{equation}\label{eq:8}
   \tilde{\psi}|_{v^\perp}: \widetilde{H}(X,\frac{\delta_v}{2n-2},\ZZ)\cong \widetilde{H}(Y, \frac{\delta_w}{2n-2},\ZZ)
\end{equation}
by restricting to $v^\perp.$ 
The Hodge structures of both sides of \eqref{eq:8} are independent of the choice of $\delta_v~,\delta_w$, according to Lemma \ref{lem:delta}.
The statement on orientation follows from the fact that both exponential operators and Eichler transvections are orientation-preserving.

Finally, we explain why \eqref{eq:8} provides evidence for Conjecture \ref{conj:main}. Beckman \cite[Definition 5.2 and Theorem 8.1]{Bec23} and Markman \cite[Theorem 12.2]{m24obs} have shown that any (untwisted) derived equivalence preserves the lattice $H^2(X,\ZZ)\oplus U$ with a \emph{twisted} Hodge structure, where the $B$-field is exactly $\frac{\delta_v}{2}$. Therefore, if a derived Torelli theorem exists for hyper-Kähler varieties of $K3^{[n]}$-type, then the Brauer class of the derived category and the $B$-field of the corresponding twisted extended Mukai lattice must differ by $\frac{\delta_v}{2}$.

\section{Equivalences via hyperholomorphic bundles}
Assume that $n\geq 2.$
In this section, we establish Theorem \ref{thm:primitive-embedding}. Building upon Markman's projectively hyperholomorphic bundles, we further deduce in Theorem \ref{thm:criterion}  a derived Torelli-type criterion characterizing derived equivalences between projective hyperk\"ahler varieties.
\subsection{Birational $K3^{[n]}$-type varieties}\label{subsec:birational}
As a first example of Conjecture \ref{conj:main}, we consider birational pairs, which serve as a key ingredient in the proof of Theorem \ref{thm:primitive-embedding}.

Let $f:X \dashrightarrow X'$ be a birational transformation between hyper-K\"ahler varieties of ${K3}^{[n]}$-type. The induced map $f_*$ is a parallel transport operator, that extends to the following commutative diagram, where $\tilde{f}_*$ is a Hodge isometry:
\[\begin{tikzcd}
{H^2(X,\ZZ)} \arrow[r, "f_*"] \arrow[d, hook'] & {H^2(X',\ZZ)} \arrow[d, hook'] \\
L(X) \arrow[r, "\tilde{f}_*"]                    & L(X')                           
\end{tikzcd}.\]
Let $v\in H^2(X,\ZZ)^\perp \subset L(X)$ be a generator and let $v':=\tilde{f}_*(v)$.  
Note that $\tilde{f}_*:(L(X),v)\rightarrow (L(X'),v')$ is an orientation-preserving isometry.

It is clear that if $\delta_v$ satisfies Lemma \ref{lem:delta}(1)(2) with respect to $v$, then $f_*(\delta_v)$ satisfies the same conditions with respect to $v'$. 
Thus, the main theorem of \cite{MSYZ24} implies that there exists a derived equivalence
\begin{equation}\label{eq:de}
    D^b(X,[k\theta_v])\cong D^b(X',[k\theta_{v'}]).
\end{equation} for all $k\in \ZZ$.
Let's examine how $\epsilon=\pm 1$ works.
By changing the generator, we obtain an orientation-reversing isometry $\tilde{f}_*:(L(X),v)\rightarrow (L(X'),-v')$. Since $\theta_{v'}=-\theta_{-v'},$ we can rewrite the above equivalence as \eqref{eq:de} as
\[D^b(X,[k\theta_v])\cong D^b(X',[-k\theta_{-v'}]),\] which is consistent with Conjecture \ref{conj:main} when $k=n$.
\subsection{Hodge parallel transport}\label{subsec:gen}
We provide a slight generalization of the results in \cite{MSYZ24}.
Let $X,~X'$ be hyper-K\"ahler varieties of ${K3}^{[n]}$-type and let
\[\rho:H^2(X,\ZZ)\to H^2(X',\ZZ)\] be a parallel transport operator preserving the Hodge structures. This implies that $X$ and $X'$ are birational by \cite[Theorem 1.3]{Markman11}. However, $\rho$ is not necessarily induced by a birational map.
Nevertheless, Markman showed in \cite[Theorem 6.18]{Markman11} that $\rho$ can be written as a composition of prime exceptional reflections and isometries induced by birational maps.

A prime exceptional reflection $R_e:H^2(X,\ZZ)\to H^2(X,\ZZ)$
is the reflection in a certain algebraic class $e$. In particular, it acts trivially on $\mathrm{Br}(X)$.
Combining this with the main result of \cite{MSYZ24}, we obtain the following theorem:
\begin{theorem}\cite{MSYZ24}\label{thm:genD}
    Let $\rho:H^2(X,\ZZ)\to H^2(X',\ZZ)$ be any parallel transport operator that preserves Hodge structures. Then for any $B\in H^2(X,\QQ)$, there is an equivalence of twisted derived categories \[D^b(X,[B])\cong D^b(X',\left[\rho(B)\right])\]
\end{theorem}

\subsection{Equivalences from hyperholomorphic bundles}
We denote by $\mathfrak{M}_\Lambda$ the moduli space of marked manifolds $(X,\eta_X)$ of~$K3^{[n]}$-type.
Let $S$ be a $K3$ surface with $\mathrm{Pic}(S)=\ZZ H$. 
Consider a primitive and isotropic Mukai vector
\[
v_0:=(r, mH, s) \in \tilde{H}(S, \ZZ),\quad  v_0^2=0
\]
with
\[
m\in \ZZ, \quad \mathrm{gcd}(r,s)=1  
\]
satisfying that 
\begin{equation}\label{VB}
    \frac{r}{\ell} \nmid \frac{1}{2}\left( \frac{mH}{\ell}\right)^2 +1, \quad \ell := \mathrm{gcd}(r,m).
\end{equation}
Clearly (\ref{VB}) guarantees that $r\geq 2$. By \cite[Lemma 1.2]{Yoshi09}, the numerical condition (\ref{VB}) implies that the moduli space $M$ consists of stable vector bundles.

Assume that $r,s$ are two coprime integers with $r\geq 2$. Assume further that the Mukai vector
\[
v_0:=(r, mH, s) \in \tilde{H}(S, \ZZ)
\]
is isotropic, \emph{i.e.}~$v_0^2=0$, and that all stable sheaves on $S$ with Mukai vector $v_0$ are stable vector bundles.
Let $M$ be the moduli space of stable vector bundles on $S$ with Mukai vector $v_0$. Then $M$ is also a $K3$ surface, and the coprime condition on $r, s$ ensures the existence of a universal rank $r$ bundle $\cU$ on $M \times S$.
There is a Hodge isometry 
\[
F_\cU: H^2(M, \QQ) \to H^2(S, \QQ),
\]
induced by $\cU$; see \cite[Corollary 7.3]{M24}.
Conjugating the BKR correspondence, we obtain a vector bundle $\cU^{[n]}$ on $M^{[n]} \times S^{[n]}$; see \cite[Lemma 7.1]{M24}. This vector bundle induces a derived equivalence
\begin{equation}\label{Deq}
\Phi_{\cU^{[n]}}: D^b(M^{[n]}) \xrightarrow{\simeq} D^b(S^{[n]}).
\end{equation}
Under the natural identification
\begin{equation}\label{Hilb}
H^2(M^{[n]}, \QQ) = H^2(M, \QQ)\oplus \QQ \delta_{M^{[n]}}, \quad H^2(S^{[n]}, \QQ) = H^2(S, \QQ)\oplus \QQ  \delta_{S^{[n]}},
\end{equation}
Markman further showed in \cite[Section 5.6]{M24} that the characteristic class of $\cU^{[n]}$ induces a rational Hodge isometry
\footnote{This rational isometry is also described later in another way in Example \ref{ex:f} of Subsection \ref{subsec:Q-HI}.}
\begin{equation}\label{eq:hil}
    F_{\cU^{[n]}}=F_\cU\oplus \mathrm{id_\delta}: H^2(M^{[n]}, \QQ) \to H^2(S^{[n]}, \QQ). 
\end{equation}
For $g \in O(\Lambda_\QQ)$, we define $\mathfrak{M}_{g}$ to be the moduli space of isomorphism classes of quadruples $(X, \eta_X, Y, \eta_Y)$ where $(X, \eta_X), (Y, \eta_Y) \in \mathfrak{M}_{\Lambda}$ are the corresponding markings and 
\[
\eta_Y^{-1}\circ g \circ \eta_X: H^2(X, \QQ) \to H^2(Y,\QQ)
\]
is a Hodge isometry sending some K\"ahler class of $X$ to a K\"ahler class of $Y$. 

We summarize the key result as follows (for a partial summary, see \cite[Theorem 1.3]{MSYZ24}), which combines the deformation of Markman's projectively hyperholomorphic bundles \cite{M24}, the twisted D-equivalence theorem \cite{MSYZ24} and \cite[Theorem 2.3]{KK23}:
\begin{theorem}[\cite{M24, KK23,MSYZ24}]\label{thm:summary}
   Let $X,~Y$ be hyper-K\"ahler varieties of $K3^{[n]}$-type and $S,~M,~r\geq 2,~\cU,~F_{\cU^{[n]}}$ be as above.
   Assume that there exist parallel transports \[\rho_1:H^2(M^{[n]},\ZZ)\to H^2(X, \ZZ),~\rho_2:H^2(S^{[n]},\ZZ)\to H^2(Y, \ZZ)\] such that the composition \[\rho_2\circ F_{\cU^{[n]}} \circ \rho_1^{-1}:H^2(X,\QQ)\to H^2(Y,\QQ)\] is a rational Hodge isometry. Then we have an
 equivalence of twisted derived categories
 \[D^b(X,\left[\rho_1(\frac{-c_1(\cU|_{M\times \{x\}})}{r}-\frac{\delta_{M^{[n]}}}{2})\right])\cong D^b(Y,\left[\rho_2(\frac{c_1(\cU|_{\{y\}\times S})}{r}-\frac{\delta_{S^{[n]}}}{2})\right]),\]
where $x\in S,~y\in M$ are arbitrary points.
\end{theorem}
\begin{proof}
    By \cite{M24}, there exist markings $\eta_{M^{[n]}} $ and $\eta_{S^{[n]}}$ for the Hilbert schemes $M^{[n]}$ and $ S^{[n]}$, respectively, which induce an isometry $g=\eta_{S^{[n]}}\circ F_{\cU^{[n]}} \circ \eta_{M^{[n]}}^{-1} \in O(\Lambda_\mathbb{Q})$, such that 
\[(M^{[n]}, \eta_{M^{[n]}}, S^{[n]}, \eta_{S^{[n]}}) \in \mathfrak{M}^0_{g},\] where $\mathfrak{M}^0_{g}$ denotes  the distinguished connected component of $\mathfrak{M}_{g}$ in \cite[Theorem 1.3]{MSYZ24}.
There are natural forgetful maps for $i=1,~2:$ \[\Pi_i: \mathfrak{M}^0_{g}\to \mathfrak{M}_\Lambda,~~(X_1,\eta_1,X_2,\eta_2)\mapsto (X_i,\eta_i).\]
Markman proved in \cite[Lemma 5.7]{M24} that $\Pi_i$ is surjective. 
Now consider marked hyper-K\"ahler varieties $(Y,\eta_Y)$ and $(X,\eta_X)$,
where \[\eta_Y:=\eta_{S^{[n]}}\circ \rho_2^{-1} {\rm ~and~}\eta_{X}=\eta_{M^{[n]}}\circ \rho_1^{-1}\]
According to the surjectivity of $\Pi_2$, 
we obtain that $(Y,\eta_Y)$ lifts to \[(X',\eta_{X'},Y,\eta_Y)\in \mathfrak{M}^0_{g}\] for some other $(X',\eta_{X'})\in \mathfrak{M}_\Lambda.$ 
In what follows, we demonstrate that $X$ and $X'$ are birational.

Since $(X',\eta_{X'},Y,\eta_Y)$ and $(M^{[n]}, \eta_{M^{[n]}}, S^{[n]}, \eta_{S^{[n]}})$ are in the same connected component and $\rho_2=\eta_Y^{-1}\circ \eta_{S^{[n]}}$ is a parallel transport, we have \[\rho':=\eta_{X'}^{-1}\circ \eta_{M^{[n]}}:H^2(M^{{[n]}},\ZZ)\to  H^2(X',\ZZ)\] is also a  parallel transport. Therefore, we obtain a parallel transport between $X$ and $X'$: \[ \rho'\circ \rho_1^{-1} : H^2(X,\ZZ)\to H^2(X',\ZZ).\]
We now verify that $\rho'\circ \rho_1^{-1}$ preserves Hodge structures: 
A direct computation shows that
\[ \rho'\circ \rho_1^{-1}=\eta_X^{-1}\circ \eta_{X'}.\]
Since  $(X',\eta_{X'},Y,\eta_Y)\in \mathfrak{M}^0_{g}$, we have
\begin{equation}\label{h1}
    \eta_Y^{-1} \circ g \circ \eta_{X'}(\sigma_{X'})=\sigma_Y.
\end{equation}
By our assumption, the isometry
\[\rho_2\circ F_{\cU^{[n]}} \circ \rho_1^{-1}:H^2(X,\QQ)\to H^2(Y,\QQ)\] preserves Hodge structures. Thus, we obtain
\begin{equation}\label{h2}
    \eta_Y^{-1}\circ g\circ \eta_X(\sigma_X)=\rho_2\circ F_{\cU^{[n]}} \circ \rho_1^{-1}(\sigma_X)=\sigma_Y.
\end{equation}
Combining Equation \ref{h1} and \ref{h2}, we conclude  that \[\eta_X^{-1}\circ \eta_{X'}(\sigma_{X'})=\sigma_X\] and hence $\rho'\circ \rho_1^{-1}=\eta_X^{-1}\circ \eta_{X'}$ preserves Hodge structures.

 
Therefore, Theorem \ref{thm:genD} together with \cite[Theorem 1.3(b,c)]{MSYZ24} yields the desired twisted derived equivalence. 
\end{proof}
\subsection{Rational Hodge isometries between $H^2$}\label{subsec:Q-HI}
Let $X$ and $Y$ be two hyper-K\"ahler varieties of ${K3}^{[n]}$-type. There is a natural way to obtain a rational Hodge isometry between $H^2$ from the integral isometry of $\tilde{H}$, which is similar to \cite[Proposition 2.1]{B19}. 
Let \[A\in H^2(X,\QQ),~B\in H^2(Y,\QQ)\] be two $B$-fields and $\psi$ be a Hodge isometry \[\psi:\tilde{H}(X,A,\ZZ)\rightarrow \tilde{H}(Y,B,\ZZ).\]
Recall that as a lattice, $\tilde{H}(X,A,\ZZ)\cong H^2(X,\ZZ)\oplus U$ with $U$ generated by $e,~f$. We define \[r=r(\psi):= -\left<\psi(f).f\right>\in \ZZ.\]
Then we have \begin{equation}\label{eq:HK-vector}
    \psi(f)=re+m\beta+sf,{\rm~and}~\psi^{-1}(f)=re+m'\alpha+s'f
\end{equation}for some primitive class
 $\alpha\in H^2(X,\ZZ),~\beta\in H^2(Y,\ZZ)$ and the same $r$, where $r,~m,~s,~s'\in \ZZ$. We have the following two results.
\begin{lemma}\label{lem:2-B-fields}
    Let $X,~Y,~A,~B,~,\psi,~r,~m,~\alpha,~,\beta$ be as above. If $r\neq 0$, then we have \[\left[\frac{m'\alpha}{r}\right]=[A]\in \mathrm{Br}(X),~\left[\frac{m\beta}{r}\right]=[B]\in \mathrm{Br}(Y).\]
\end{lemma}
\begin{proof}
    Since $\psi$ is Hodge, we have $\psi(f)=re+m\beta+sf$ is a Hodge class, which implies that \[\left<m\beta-rB.\sigma_Y\right>=\left<m\beta-rB.\Bar{\sigma}_Y\right>=0.\] Hence \[\left[\frac{m\beta}{r}-B\right]=1\in  \mathrm{Br}(Y).\] The other equation is similar.
\end{proof}
\begin{proposition}\label{prop:diag}
    If $r\neq 0$, then the rational isometry ${\bfe}^{\frac{-m\beta}{r}}\circ \psi \circ {\bfe}^{\frac{m'\alpha}{r}}$ acts diagonally. More precisely, we have the following commutative diagram
    \begin{equation*}
        \begin{tikzcd}
{\tilde{H}(X,A,\ZZ)} \arrow[rr, "\psi"] \arrow[d, "{\bfe}^{-\frac{m'\alpha}{r}}"] &  & {\tilde{H}(Y,B,\ZZ)} \arrow[d, "{\bfe}^{-\frac{m\beta}{r}}"] \\
{H^2(X,\QQ)\oplus U_\QQ} \arrow[rr, "F_\psi\oplus \rho_{f-re}"]                      &  & {H^2(Y,\QQ)\oplus U_\QQ} .                                      \end{tikzcd}
    \end{equation*}
    Here, $F_\psi$ is a rational isometry and $\rho_{f-re}$ is the reflection in the vector $f-re.$
\end{proposition}
\begin{proof}
      We can check that ${\bfe}^{\frac{-m\beta}{r}}\circ \psi \circ {\bfe}^{\frac{m'\alpha}{r}}$ sends $e$ to $\frac{f}{r}$ and $f$ to $re$. Hence we also have \[{\bfe}^{\frac{-m\beta}{r}}\circ \psi \circ {\bfe}^{\frac{m'\alpha}{r}}(U_\QQ^\perp)=(U_\QQ^\perp).\] Then \[{\bfe}^{\frac{-m\beta}{r}}\circ \psi \circ {\bfe}^{\frac{m'\alpha}{r}}= F_\psi\oplus \rho_{f-re}\] for some isometry $F_\psi$. 
       To verify that $F_\psi$ is a Hodge isometry,  it suffices to observe the identity  \[\psi({\bfe}^{\frac{m'\alpha}{r}}(\sigma_X))={\bfe}^{\frac{m\beta}{r}}(\sigma_Y),\] which follows from Lemma \ref{lem:2-B-fields} and the fact that $\psi$ is a Hodge isometry between the twisted extended Mukai lattices.
\end{proof}
In conclusion, we have associated a rational Hodge isometry \begin{equation}\label{eq:def-f}
    F_\psi:H^2(X,\QQ)\to H^2(Y,\QQ)
\end{equation}
to any $\psi:\tilde{H}(X,A,\ZZ)\rightarrow \tilde{H}(Y,B,\ZZ)$ with $r(\psi)\neq 0.$ Geometrically, this corresponds to changing Chern class to $\kappa$ class,  as discussed in \cite[Introduction]{M24} and \cite[Proposition 2.1]{B19}. 
We have the following important example.
\begin{example}\label{ex:f}
    Let $S,~M$ be two $K3$ surfaces and there is a universal sheaf $\cU$ on $S\times M$, which induces a derived equivalence.
    Then we have Hodge isometries \[\psi_\cU:\tilde{H}(M,\ZZ)\to \tilde{H}(S,\ZZ),~ \psi_{\cU^{[n]}}:\tilde{H}(M^{[n]},\ZZ)\to \tilde{H}(S^{[n]},\ZZ).\]
    Then we can identify the isometry in \eqref{eq:hil} with the induced isometry in this subsection (c.f. \cite[Proposition 2.1]{B19} and \cite[Corollary 7.3]{M24}): \[F_{\cU}=F_{\psi_\cU},~ F_{\cU^{[n]}}=F_{\psi_{\cU^{[n]}}}.\]
    Moreover, since the inverse of $\psi_\cU$ is induced by $\cU^\vee$, we obtain 
    $$\psi_\cU(f)=(r,c_1(\cU),s){~\rm and~}\psi^{-1}_\cU(f)=(r,-c_1(\cU),s').$$
\end{example}

\subsection{A criterion of equivalence}
We now present a criterion for lifting an integral Hodge isometry between twisted extended Mukai lattices to a twisted derived equivalence.
Let the following data be as in the previous subsection:
\begin{itemize}
    \item a Hodge isometry $\psi:\tilde{H}(X,A,\ZZ)\rightarrow \tilde{H}(Y,B,\ZZ)$,
    \item $\psi(f)=re+m\beta+sf,{\rm~and}~\psi^{-1}(f)=re+m'\alpha+s'f$ with $r\neq 0$,
    \item the induced rational Hodge isometry $F_\psi:H^2(X,\QQ)\to H^2(Y,\QQ)$.
\end{itemize}
Denote by $ \Lambda_{K3}$ the unimodular $K3$-lattice. Let \[\Lambda = \Lambda_{K3} \oplus \ZZ \delta, ~ \delta^2 = 2-2n\] be the $K3^{[n]}$-lattice,
   which is isometric to $H^2(X, \ZZ)$, equipped with the BBF form.
Recall that a $\delta$-class is a class satisfying Lemma \ref{lem:delta}(1)(2).  
\begin{lemma}\label{lem:parallel}
    Let $\rho: H^2(X,\ZZ)\rightarrow H^2(Y,\ZZ)$ be an orientation-preserving integral isometry. Assume that  $\rho(\delta_v)=\delta_w$ for some $\delta$-classes. Then $\rho$ is a parallel transport operator.
\end{lemma}
\begin{proof}
    The proof is essentially in \cite[Chapter 14, Proposition 2.6]{Huy16}. The isometry $\rho$ admits a natural rational lift  $\tilde{\rho}:L(X)_\QQ\to L(Y)_\QQ$, where \[\tilde{\rho}|_{H^2(X,\ZZ)}=\rho,~\tilde{\rho}(v)=w.\]
    By \cite[Theorem 9.8]{Markman11}, it suffices to show that $\tilde{\rho}$ is an integral isometry.
    Note that we have (and similarly for $Y$) \[H^2(X,\ZZ)^\vee/H^2(X,\ZZ)\cong \ZZ \frac{\delta_v}{2n-2},~ {\rm and ~} (\ZZ v)^\vee/\ZZ v \cong \ZZ\frac{v}{2n-2}.\] Since $\rho(\delta_v)=\delta_w$ and $\tilde{\rho}(v)=w$, it suffices to verify the following
    \begin{equation}\label{eq:parallel}
        \frac{k\delta_v+\ell v}{2n-2}\in L(X)\Leftrightarrow  \frac{k\delta_w+\ell w}{2n-2}\in L(Y),
    \end{equation}
    where $k,~\ell$ are integers. By Lemma \ref{lem:delta}(2), both $\frac{\delta_v-v}{2n-2}$ and $\frac{\delta_w-w}{2n-2}$ are integral classes. Therefore, \eqref{eq:parallel} is equivalent to \[\frac{(\ell+k)v}{2n-2}\in \ZZ v \Leftrightarrow \frac{(\ell+k)w}{2n-2}\in \ZZ w,\] which is obvious.

\end{proof}
We emphasize that the rank $r=r(\psi)\neq 0$ and begin by addressing an exceptional case.
\begin{lemma}\label{lem:-2ref}
     Let the settings be as above.
     Assume that ${\rm gcd}(r,m)=\ell$, which allows us to write \[\psi(f)=\ell(r_0+m_0\beta)+sf\] for coprime integers $r_0$ and $m_0$. Assume further that \[r_0|\frac{m_0^2\beta^2}{2}+1.\] Then we have a Hodge isometry  \[\varphi: H^2(X,\ZZ)\cong H^2(Y,\ZZ). \] 
\end{lemma}
\begin{proof}
    We set \[t:=\frac{m_0^2\beta^2+2}{2r_0}\in \ZZ \] and then \[u:=r_0e+m_0\beta+tf\in \tilde{H}(Y,B,\ZZ) {~\rm satisfies~} u^2=-2. \] 
    Since $\psi(f)^2=0$, we have \[m_0^2\ell\frac{\beta^2}{2}=r_0s.\]
    Then the assumption $r_0|\frac{m_0^2\beta^2}{2}+1$ yields $r_0|\ell$. Furthermore, since ${\rm gcd}(r,s)={\rm gcd}(r_0\ell,s)=1$, we obtain $s|m_0^2\frac{\beta^2}{2}$.  It forces that \[r_0=\pm \ell,~s=\pm m_0^2\frac{\beta^2}{2}.\]
     Now let $\rho_{u}:\tilde{H}(Y,B,\ZZ)\to \tilde{H}(Y,B,\ZZ)$ be the reflection in the vector $u$. We can compute that \[\rho_{u}\circ \psi(f)=\mp f.\]
     Therefore, $\psi':=\rho_{u}\circ \psi$ induces the following Hodge isometry: \[\varphi: H^2(X,\ZZ)\cong f^\perp/f\xrightarrow[\simeq]{\rho_{u}\circ \psi} f^\perp/f  \cong H^2(Y,\ZZ),\] which concludes the proof. 
\end{proof}

\begin{theorem}\label{thm:criterion}
    Let the settings be as above. 
    Suppose that $\psi:\tilde{H}(X,A,\ZZ)\rightarrow \tilde{H}(Y,B,\ZZ)$ is orientation-preserving. If there exists a $\delta$-class $\delta_v\in H^2(X,\ZZ)$ such that \begin{equation}\label{eq:jiashe}
       \delta_w:=\psi(\delta_v)\in H^2(Y,\ZZ),
    \end{equation}
    and that $\delta_w$ is a $\delta$-class with respect to $w$,
    then we have a derived equivalence \footnote{Since $\left[\frac{\delta_v}{2}\right]=\left[-\frac{\delta_v}{2}\right]$, the sign in front of $\frac{\delta_v}{2}$ does not matter.}
    \footnote{In this case, the induced rational Hodge isometry $F_\psi$ is of cyclic type in the sense of \cite{M24} and the isometry $\psi$ can be lifted to $L(X)\oplus U\to L(Y)\oplus U$ with a suitable Hodge structure.}
    \[D^b(X,\left[A+\frac{\delta_v}{2}\right])\cong D^b(Y,\left[B+\frac{\delta_w}{2}\right])\]
\end{theorem}
\begin{proof}
   Our goal is to find a $K3$ surface $S$, an isotropic Mukai vector $v_0:=(r, mH, s)\in \tilde{H}(S,\ZZ)$, along with appropriate parallel transport operators, so that Theorem \ref{thm:summary} can be applied.
   By Equation \eqref{eq:HK-vector} and \eqref{eq:jiashe}, 
   we have 
   \begin{equation}\label{eq:div-1}
\left<\delta_v.\alpha\right>=0,~\left<\delta_w.\beta\right>=0,~F_\psi(\delta_v)=\delta_w.
   \end{equation}
   We treat the following three cases separately.
\medskip

 \noindent {\bf Case 1:  $r=\pm 1$.}
In this case, $F_\psi$ is an integral Hodge operator and preserves the orientation. Furthermore, it is a parallel transport operator by Lemma \ref{lem:parallel} and \eqref{eq:div-1}. Hence $X$ is birational to $Y$.
Moreover, by Lemma \ref{lem:2-B-fields}, both $[A]$ and $[B]$ are trivial. Then the assertion follows from Subsection \ref{subsec:birational}.

\noindent {\bf Case 2: $r\neq\pm 1$, general case.}
We assume that the phenomenon described in Lemma \ref{lem:-2ref} does not occur; that is,
   \begin{equation}\label{eq:yos}
       r_0\nmid \frac{m_0^2\beta^2}{2}+1.
   \end{equation}
The core idea is to construct $K3$ surfaces satisfying the conditions of Theorem \ref{thm:summary} and to check the corresponding Brauer classes. Both steps crucially depend on the assumption \eqref{eq:jiashe} concerning $\delta$-classes.  We emphasize that throughout the proof, all constructed isometries between the second cohomology groups $H^2$ preserve the various $\delta$.

\noindent{\it Step 1:} 
   If $r<0$, we consider $-\psi$ instead of $\psi.$ Therefore, we can assume that $r\geq 2$. 
 Take markings $\eta_X~,\eta_Y$ that satisfy \[\eta_X(\delta_v)=\delta,~\eta_Y(\delta_w)=\delta\]
   and thus we have the following commutative diagram 
   \[\begin{tikzcd}
{H^2(X,\QQ)} \arrow[d, "\eta_X"] \arrow[r, "F_\psi"]                          & {H^2(Y,\QQ)} \arrow[d, "\eta_Y"]              \\
{\Lambda_{K3,\QQ} \oplus \QQ \delta} \arrow[r, "g\oplus \id_\delta"] & {\Lambda_{K3,\QQ} \oplus \QQ \delta}
\end{tikzcd},\] for some isometry $g:\Lambda_{K3,\QQ}\to \Lambda_{K3,\QQ}$.
Due to \eqref{eq:div-1}, we have \[\eta_X(\alpha)\in \Lambda_{K3},~ \eta_Y(\beta)\in \Lambda_{K3}.\]

\noindent{\it Step 2:} The above isometry $g:\Lambda_{K3,\QQ}\to \Lambda_{K3,\QQ}$, combined with the global Torelli theorem, allows us to construct the $K3$ surfaces $S$ and $M$.
In this step, we select a distinguished $K3$ surface $S$;  the technical requirement is a coprimality condition on integers (imposed below), which will ultimately ensure in the next step that  $M$ is a fine moduli space of objects on $S$.

Take an integral vector $\gamma\in \Lambda_{K3}$ such that \[\left<\eta_Y(\beta).\gamma\right>=1.\] Recall that $\psi(f)=re+m\beta+sf$ is an isotropic vector. Let \[v_k:={\bfe}^{k\gamma}(re+m\eta_Y(\beta)+s)=re+(m\eta_Y(\beta)+rk\gamma)+(s+km+rk^2\cdot\frac{\gamma^2}{2})f\in \Lambda_{K3}\oplus U.\]
Define \[s_k:=s+km+rk^2\cdot\frac{\gamma^2}{2}.\] We claim that there exists $k\in \ZZ$ such that \[\mathrm{gcd}(r,s_k)=1.\]
In fact, we can take $k$ such that $\frac{s+km}{\mathrm{gcd}(s,m)}$ to be a large prime number and thus we have \[\mathrm{gcd}(r,s_k)=\mathrm{gcd}(r,s+km)=\mathrm{gcd}(r,\frac{s+km}{\mathrm{gcd}(s,m)})=1.\] The second equality is due to the fact that $re+m\beta+sf$ is primitive.
We write  \[\beta_k:= m\eta_Y(\beta)+rk\gamma.\]
We can find a $K3$ surface $S$ along with a marking \[\eta_S:H^2(S,\ZZ)\rightarrow \Lambda_{K3},~H\mapsto \beta_k\] such that $H\in \Pic(S)$, although $H$ is not necessarily primitive. By the surjectivity of the period map of $K3$ surfaces, we can find $(M,\eta_M)$ and $h$ with the following commutative diagram.
\[
\begin{tikzcd}
{H^2(M^{[n]},\QQ)} \arrow[d, "{\eta_{M^{[n]}}}"] \arrow[r, "h"]               & {H^2(S^{[n]},\QQ)} \arrow[d, "{\eta_{S^{[n]}}}"] \\
{\Lambda_{K3,\QQ} \oplus \QQ \delta} \arrow[r, "g\oplus \id_\delta"] & {\Lambda_{K3,\QQ} \oplus \QQ \delta},  
\end{tikzcd}
\]
where $\eta_{M^{[n]}}=\eta_M\oplus \id_\delta,~\eta_{S^{[n]}}=\eta_S\oplus \id_\delta$, and $h$ is a rational Hodge isometry. Here, by abuse of notation, $\id_\delta$ means sending $\delta_{M^{[n]}},~\delta_{S^{[n]}}$ to $\delta.$

\noindent{\it Step 3:} We now demonstrate that $M$ is a fine moduli space of stable vector bundles on $S$ with isotropic Mukai vector \[v_0:=(r, H,s_k)\in \tilde{H}(S,\ZZ).\] Let  $M_{v_0}$ denote this moduli space and define the operators \[\rho_1=\eta^{-1}_X\circ \eta_{M^{[n]}},~\rho_2=\eta^{-1}_Y\circ \eta_{S^{[n]}} .\] By Lemma \ref{lem:parallel},  $\rho_1,~\rho_2$ are parallel transport operators up to replacing $\eta_S,~\eta_M$ with $-\eta_S,~-\eta_M$.

To establish  $M_{v_0}\cong M$, we conjugate the diagram in Proposition \ref{prop:diag} via these parallel transports, yielding an integral Hodge isometry \begin{equation}\label{eq:psi'}
    \psi':\tilde{H}(M^{[n]},-\frac{m' \rho_1^{-1}(\alpha)}{r},\ZZ)\rightarrow \tilde{H}(S^{[n]},\ZZ),
\end{equation} where \[ \psi'={\bfe}^{\frac{m\rho_2^{-1}(\beta)}{r}+k\eta^{-1}_S(\gamma)}\circ (h\oplus \rho_{f-re}) \circ {\bfe}^{-\frac{m' \rho_1^{-1}(\alpha)}{r}}\]
and we have $h=F_{\psi'}$.
The Brauer class on $S^{[n]}$ vanishes because  \[\frac{\rho^{-1}_2(m\beta)}{r}=\frac{H}{r}-k\eta^{-1}_S(\gamma){\rm ~and~}H \in \Pic(S).\] 
Combining \eqref{eq:div-1} with our construction, we establish $\psi'(\delta_{M^{[n]}})=\delta_{S^{[n]}}$ and $\rho_1^{-1}(\alpha)\in \delta_{M^{[n]}}^\perp$. 
Restricting $\psi'$ to $\delta_{M^{[n]}}^\perp$, we obtain a Hodge isometry \[\psi'|_{\delta_{M^{[n]}}^\perp}:\tilde{H}(M,-\frac{m' \rho_1^{-1}(\alpha)}{r},\ZZ)\rightarrow \tilde{H}(S,\ZZ).\]
A direct computation confirms
\begin{equation}\label{eq:mukai}
    \psi'|_{\delta_{M^{[n]}}^\perp}(f)={\bfe}^{k\eta_S^{-1}(\gamma)}(r,m\rho_2^{-1}(\beta),s)=(r,H,s_k)=v_0\in \tilde{H}(S,\ZZ).
\end{equation}
Consequently, \[H^2(M)\cong f^\perp/f\cong v_0^\perp/v_0\cong H^2(M_{v_0}),\]
which implies
\[M\cong M_{v_0}.\] Moreover, the coprimality condition $\mathrm{gcd}(r,s_k)=1$ guarantees the existence of a universal sheaf $\cU$ on $M\times S$.
Finally, we have to show that $M$ consists of vector bundles on $S$. Write \[v_0=(\ell r_0,\ell H_0,s),~\ell r_0=r,~\ell H_0=H\] such that $(r_0,H_0,0)$ is primitive. Then by \eqref{eq:yos}, we have \[r_0\nmid \frac{(m_0\beta+r_0k\eta_Y^{-1}(\gamma))^2}{2}=\frac{H_0^2}{2}+1,\]
which by \cite[Lemma 1.2]{Yoshi09} implies that all stable sheaves on $S$ with Mukai vector $v_0$ are stable vector bundles.

As a summary, we obtain the following commutative diagram:
\[\begin{tikzcd}
{{H^2(M^{[n]},\QQ)}} \arrow[r, "F_{\psi'}\oplus \id_\delta"] \arrow[d, "\rho_1"] & {{H^2(S^{[n]},\QQ)}} \arrow[d, "\rho_2"] \\
{H^2(X,\QQ)} \arrow[r, "F_{\psi}"]                              & {H^2(Y,\QQ)},                           
\end{tikzcd}\]
Where $M$ is a fine moduli space of stable vector bundles on $S$. 

\noindent{\it Step 4:}
The universal sheaf $\cU$ gives another Hodge isometry:
\[\psi_\cU:\tilde{H}(M,\ZZ)\to \tilde{H}(S,\ZZ), \]
which induces another rational isometry \[F_{\psi_\cU}:H^2(M,\QQ)\to H^2(S,\QQ).\]
In prior, the isometry $\psi_\cU$ might not coincide with $\psi'$.
To apply Theorem \ref{thm:summary}, we must modify the parallel operators to replace $F_{\psi'}\oplus \id_\delta$ with $F_{\psi_\cU}\oplus \id_\delta$.

Due to \eqref{eq:mukai}, we obtain the following
 \[\psi^{-1}_\cU\circ \psi'(f)=f.\]
It yields that \[\psi_\cU^{-1}\circ \psi'(e)={\bfe}^L(e)\] for some $L\in H^2(M,\ZZ)$. It follows that
\begin{equation}\label{eq:two-iso}
    \psi'=  \psi_\cU \circ {\bfe}^L\circ\varphi,
\end{equation} where  $\varphi:H^2(M,\ZZ)\to H^2(M,\ZZ)$ is a Hodge isometry. Now we can compute that \[F_{\psi'}= F_{\psi_\cU}\circ\varphi.\]
Then we can replace $\rho_1$ with \[\rho_1':=\rho_1\circ (\varphi^{-1}\oplus \id_\delta)\] 
and we obtain the desired commutative diagram. 
\[\begin{tikzcd}
{{H^2(M^{[n]},\QQ)}} \arrow[r, "F_{\psi_\cU}\oplus \id_\delta"] \arrow[d, "\rho_1'"] & {{H^2(S^{[n]},\QQ)}} \arrow[d, "\rho_2"] \\
{H^2(X,\QQ)} \arrow[r, "F_{\psi}"]                              & {H^2(Y,\QQ)},                           
\end{tikzcd}\]

\noindent{\it Step 5:} Now we can check the Brauer classes.
We have constructed the ingredients $S,~v_0,M,~\rho'_1,~\rho_2$ in Theorem \ref{thm:summary}. 
Finally, the Brauer classes  
in the conclusion are due to  
Lemma \ref{lem:2-B-fields}, Lemma \ref{lem:delta}, and  
Brauer classes in Theorem \ref{thm:summary}.
For the $Y$ side, the Brauer class is 
\[\left[\rho_2(\frac{H}{r}-\frac{\delta_{S^{[n]}}}{2})\right]=\left[\frac{m\beta+rk\eta_X^{-1}(\gamma)}{r}-\frac{\delta_w}{2}\right]=\left[B+\frac{\delta_w}{2}\right].\]
For the $X$ side, note that by Example \ref{ex:f} and Equation \eqref{eq:psi'}, we have \[\psi'^{-1}(f)=re-m'(c_1(\cU|_{M\times \{x\}}))+s'f=re+m'\rho_1^{-1}(\alpha)+s'f\] for any $x\in S$. It implies that \[\rho_1^{-1}(\alpha)=-c_1(\cU|_{M\times \{x\}}).\]
Moreover, due to \eqref{eq:two-iso}, we have \[\psi^{-1}_{\cU^{[n]}}(f)=re+(-m'(\varphi\oplus \id)\circ \rho_1^{-1}(\alpha)+rL)+s''f\] for some integer $s''$. Thus, the Brauer class is
\[\left[\rho_1'(\frac{m'\rho_1'^{-1}(\alpha)+rL}{r}-\frac{\delta_{M^{[n]}}}{2})\right]=\left[\frac{m'\alpha}{r}-\frac{\delta_v}{2}\right]=\left[A+\frac{\delta_v}{2}\right].\]
This case is now concluded by Theorem \ref{thm:criterion}.

 \noindent {\bf Case 3: $r\neq\pm 1$, exceptional case. }
We assume that \[r_0\mid \frac{m_0^2\beta^2}{2}+1\] as in Lemma \ref{lem:-2ref}, which yields a $(-2)$-vector \[u=(r_0e+m_0\beta +tf)\in \tilde{H}(Y,B,\ZZ)\] such that \[\rho_{u}\circ \psi(f)=f\] up to composing $-\id$.
Let $\varphi$ be the Hodge isometry constructed in Lemma \ref{lem:-2ref}.  By an argument similar to that in Case 2, we have
\begin{equation}\label{eq:ftof}
    \psi':=\rho_{u}\circ \psi=(\varphi\oplus \id_U)\circ {\bfe}^L
\end{equation} for some $L\in H^2(Y,\ZZ).$
By the assumption $\delta_w\in H^2(Y,\ZZ)$ and \eqref{eq:div-1}, we have \[\left<u.\delta_w\right>=\left<m_0\beta.\delta_w\right>=0,\] which together with \eqref{eq:jiashe} yields \[\psi'(\delta_v)=\psi(\delta_v)=\delta_w.\] It forces \[ \left<L.\delta_w \right>=0~{\rm and~}\varphi(\delta_v)=\delta_w.\]
Then, by Lemma \ref{lem:parallel}, the Hodge isometry $\varphi$ constructed in Lemma \ref{lem:-2ref} is a parallel transport operator. Moreover, Equation \eqref{eq:ftof} implies that \[[B]=[\varphi(A)+L]=[\varphi(A)].\] Therefore, we obtain the desired derived equivalence by Theorem~\ref{thm:genD}.

The proof of Theorem \ref{thm:criterion} is now complete. 
\end{proof} 
\begin{remark}\label{rem:dual}
    If $\psi:\tilde{H}(X,A,\ZZ)\rightarrow \tilde{H}(Y,B,\ZZ)$ is orientation-reversing, we obtain \[D^b(X,\left[A+\frac{\delta_v}{2}\right])\cong D^b(Y,\left[-B+\frac{\delta_w}{2}\right]).\] The reason is that we can compose the following orientation-reversing Hodge isometry:
    \[\id_U\oplus -\id_{H^2(Y,\ZZ)}:\tilde{H}(Y,B,\ZZ)\to \tilde{H}(Y,-B,\ZZ).\]
\end{remark}

\subsection{Proof of Theorem \ref{thm:primitive-embedding}}
Let $X$ and $Y$ be two hyper-K\"ahler varieties of $K3^{[n]}$-type with $n\geq 2$, and let \[\phi:(L(X),v)\rightarrow (L(Y),w)\] be a Hodge isometry.
If $\phi(v)=w$, then $X$ is birational to $Y$, and the result follows from Subsection \ref{subsec:birational}. We therefore assume that $\phi(v)\neq w$. 

Since $v^2=w^2=2n-2\neq 0$, at least one of \[(\phi(v)-w)^2 {~\rm and ~}(\phi(v)+w)^2\] is nonzero. By composing $ \phi $ with the lattice automorphism $ (L(Y), w) \rightarrow (L(Y), -w) $, we may assume  
\[
\left< \phi(v). (\phi(v) - w) \right> = \frac{1}{2}(\phi(v) - w)^2 \neq 0.
\]  
The orientation is not essential to the proof since we have $\theta_v=-\theta_{-v}\in H^2(X,\mu_{2n-2})$.

Recall from Subsection \ref{subsec:proof-lattice}, the Hodge isometry \eqref{eq:5} 
\[
    \tilde{\psi}:{\bfe}^{\frac{\delta_w-w}{2n-2}} \circ E_{\tilde{\phi}(v)-w}\circ \tilde{\phi}\circ  {\bfe}^{\frac{v-\delta_v}{2n-2}} :L(X)\oplus U\to L(Y)\oplus U
\] satisfies $\tilde{\psi}(v)=w$.
Restricting $ \tilde{\psi} $ to $v^\perp $, we obtain a Hodge isometry  
\[\psi:=\tilde{\psi}|_{v^\perp}: \widetilde{H}(X,\frac{\delta_v}{2n-2},\ZZ)\cong \widetilde{H}(Y, \frac{\delta_w}{2n-2},\ZZ).\]
Let $r=r(\psi)$ and we can compute that \[r=r(\psi)=-\left<\psi(f).f\right>=\frac{1}{2}(\phi(v)-w)^2\neq 0.\] 
The rational Hodge isometry \( F_\psi \) from Proposition \ref{prop:diag} then takes the explicit form 
\begin{equation}\label{eq:F_psi}
    F_\psi=(\rho_{\phi(v)-w}\circ \phi)|_{H^2(X,\QQ)}:H^2(X,\QQ)\rightarrow H^2(Y,\QQ),
\end{equation}
where $\rho_{\phi(v)-w}:L(Y)_\QQ\rightarrow L(Y)_\QQ$ is the reflection in $\phi(v)-w$.

The construction of \( \psi \) depends on choices of \(\delta\)-classes. We impose the following:  
\begin{equation}\label{eq:key}
   {\bf Assumption :} \left<\phi(\delta_v).w\right>=-r {\rm~ for ~some ~choice~ of~} \delta_v. \tag{$\dagger$}
\end{equation}
Under this assumption, the class \[\phi(\delta_v)-\phi(v)+w\] lies in $w^\perp =H^2(Y,\ZZ)$ and satisfies the axioms of a $\delta$-class (Lemma \ref{lem:delta}(1)(2)).
Setting  \[\delta_w:=\phi(\delta_v)-\phi(v)+w,\] we obtain \[\psi(\delta_v) = \delta_w ,\] and Theorem \ref{thm:criterion} delivers the conclusion.  

Next,  we prove that we can choose $\delta_v$ to satisfy the Assumption \eqref{eq:key} when \[{\rm Span}\left(\phi(v),w\right)\subset L(Y)\] is a primitive lattice embedding. We choose two hyperbolic planes \[U_1\oplus U_2\subset L(Y)\] with generators $e_1,~f_1$ and $e_2,~f_2$. Here, $e_i^2=f_i^2=0,~\left<e_i.f_i\right>=-1.$
There is another primitive lattice embedding \[{\rm Span}(e_1-(n-1)f_1,e_1-(n-1-r)f_1-e_2+rf_2)\subset L(Y).\] 
The assignment 
\[g'(\phi(v))=e_1-(n-1)f_1,~g'(w)=e_1-(n-1-r)f_1-e_2+rf_2\] defines an isometry of primitive sublattices.
Then by the main theorem of \cite{james} or \cite[Chapter 14, Corollary 1.9]{Huy16}, the isometry $g'$ between two primitive sublattices extends to an isometry $g\in O(L(Y))$. In particular, we still have \[g(\phi(v))=e_1-(n-1)f_1,~g(w)=e_1-(n-1-r)f_1-e_2+rf_2.\] We can then take \[\delta_v=\phi^{-1}(g^{-1}(e_1+(n-1)f_1)).\] 
A direct calculation confirms   \[\left<\phi(\delta_v).w\right>=\left<(e_1+(n-1)f_1.(e_1-(n-1-r)f_1-e_2+rf_2)\right>=-r,\]
verifying Assumption \eqref{eq:key}.

Now we complete the proof of Theorem \ref{thm:primitive-embedding}.

\section{Proof of Theorem \ref{thm:k3}}
 From now on, we fix a projective $K3$ surface $S$ and $H\in \mathrm{Pic}(S)$ be a primitive class (not necessarily ample) with $H^2=2d$.
The Markman-Mukai lattice of any moduli space of stable objects is identified with $\widetilde{H}(S,\ZZ)$, although  the orientation given by Subsection \ref{subsec:oriented} might be different from the classical one.
 Let $n>1$ and let \[v=(1,0,1-n),~ w=(r,kH,s)\in \widetilde{H}^{1,1}(S,\ZZ)\] be two primitive Mukai vectors with the same square for some $k\in \ZZ$. 
Then we have
\begin{equation}\label{eq:norm}
    \frac{w^2}{2}=k^2d-rs=n-1>0.
\end{equation}
The strategy is to find a Hodge isometry \[\phi: \widetilde{H}(S,\ZZ)\rightarrow \widetilde{H}(S,\ZZ)\] such that \[{\rm Span}\left(\phi(w),v\right)\subset \widetilde{H}(S,\ZZ)\] is a primitive lattice embedding. Then Theorem \ref{thm:primitive-embedding} will imply Theorem \ref{thm:k3}.
 \subsection{Primitive embedding}
We now give a criterion of primitive lattice embedding.
\begin{lemma}\label{lem:primitive}
Assume that $w\neq v$.
    The lattice embedding $\left<w,v\right>\subset \widetilde{H}(S,\ZZ)$ is  primitive if $\mathrm{gcd}\left((r^2-1)s,k\right)=1$.
\end{lemma}
\begin{proof}
    It suffices to show that if a linear combination of $v$ and $w$ with rational coefficients is an integral class in $\widetilde{H}(S,\ZZ)$, then the coefficients must also be integral.
    Assume that \[\frac{1}{k}(xw+yv)\in \widetilde{H}(S,\ZZ){\rm ~for ~some}~x,~y\in\QQ.\] Since the classes $v,~,w$ and $H$ are all primitive, we have \[x,~y\in\ZZ,~k\mid rx+y,~k\mid sx+(1-n)y.\]
    By taking a linear combination, we obtain \[k\mid (r(n-1)+s)x.\]
    Applying \eqref{eq:norm} to substitute $n-1$
    , we get \[k\mid (1-r^2)sx.\] The if $\mathrm{gcd}\left((r^2-1)s,k\right)=1$ we must have $k\mid x$. Then $k$ also divides $y$ since $v$ is primitive. 
\end{proof}
\subsection{End of proof}
If $r=\pm 1$, then ${\bfe}^{\pm kH}(w)=\pm v$ for some $k$ and the result follows from Subsection \ref{subsec:birational}.
For $r\neq \pm 1$, we complete the proof by showing that there exists an integer $t$, such that \[{\rm Span}\left({\bfe}^{tH}(w),v\right)\subset \widetilde{H}(S,\ZZ)\] is a primitive lattice embedding. We compute that \[{\bfe}^{tH}(w)=(r,(k+rt)H,s+2ktd+rt^2d).\]
\begin{lemma}\label{lem:t}
    There exists $t\in \ZZ$ such that \[\mathrm{gcd}\left((r^2-1)(s+2ktd+rt^2d),k+rt\right)=1.\]
\end{lemma}
\begin{proof}
    We define $\ell: = \mathrm{gcd}\left(k,r\right)$, which allows us to write
\[r=r'\ell,~k=k'\ell,~\mathrm{gcd}\left(k',r'\right)=1.\] Since $\mathrm{gcd}\left(r,k,s\right)=1$ and $\ell\mid r$, we have \[\mathrm{gcd}\left(s,\ell\right)=1,~\mathrm{gcd}\left(r^2-1,\ell\right)=1.\]
Therefore, there exists an infinite set $I\subset \ZZ$ such that for any $t\in I$, \[k'+r't {\rm ~is ~a ~prime~number~and~} \mathrm{gcd}\left(r^2-1,k+rt\right)=1.\] We now obtain \[\mathrm{gcd}\left((r^2-1)(s+2ktd+rt^2d),k+rt\right)=\mathrm{gcd}\left(s+ktd,k'+r't\right)\] for any $t\in I$. 

We claim that there exists $t\in I$ such that $\mathrm{gcd}\left(s+ktd,k'+r't\right)=1$. If not, for $t\in I$ sufficiently large, we must have \[\frac{s+ktd}{k'+r't}=\frac{kd}{r'}\in \ZZ\] 
since $k'+r't$ is a prime number for any $t\in I$.
This leads to \[rs=k^2d,\] which contradicts \eqref{eq:norm}.

\end{proof}
In conclusion, there is a Hodge isometry ${\bfe}^{tH}: \widetilde{H}(S,\ZZ)\to \widetilde{H}(S,\ZZ)$ such that
\[{\rm Span}\left({\bfe}^{tH}(w),v\right)\subset \widetilde{H}(S,\ZZ)\] is a primitive lattice embedding by Lemma \ref{lem:primitive} and Lemma \ref{lem:t}. 
Then we can apply Theorem \ref{thm:primitive-embedding} to the triples $(v,w,{\bfe}^{tH})$ and $(-v,w,{\bfe}^{tH})$. Then the proof of Theorem \ref{thm:k3} is complete.
\qed

Finally, we discuss Corollary \ref{cor:fine}.
When $n=2$, the assertion follows from Theorem \ref{thm:k3} directly. When $M_v$ is birational to a fine moduli space of stable sheaves, the result is implied by Theorem \ref{thm:k3}, Lemma \ref{lem:obs} and Subsection \ref{subsec:birational}.
\qed
\bibliographystyle {plain}
\bibliography{ref}
\end{document}